\documentclass[rm]{birkmult}
 \newtheorem{thm}{Theorem}[section]

 \newtheorem{prop}[thm]{Proposition}
 \theoremstyle{definition}
 
 \theoremstyle{remark}
 \newtheorem{rem}[thm]{Remark}
 
 \numberwithin{equation}{section}

\newcommand{\X}{\mathfrak{X}}
\newcommand{\s}{\mathfrak{S}}
\newcommand{\g}{\mathfrak{g}}
\newcommand{\W}{\mathcal{W}}
\newcommand{\norm}[1]{\left\Vert#1\right\Vert ^2}
\newcommand{\nJ}{\norm{\nabla J}}
\newcommand{\ntJ}{\norm{\tilde{\nabla} J}}

\newcommand{\thmref}[1]{Theorem~\ref{#1}}
\newcommand{\propref}[1]{Proposition~\ref{#1}}

\begin{document}
%
%
\firstpage{1}
\DOI{006-xxxx-y}
%
%
%
%
\title[Quasi-K\"ahler manifolds with a pair of Norden metrics]
{Quasi-K\"ahler manifolds \\ with a pair of Norden metrics}
%
%
%
\author[D.~Mekerov]{Dimitar Mekerov}
\address{%
  Faculty of Mathematics and Informatics \\
  University of Plovdiv  \\
  236, Bulgaria Blvd. \\
  Plovdiv 4003, Bulgaria
}
\email{mircho@pu.acad.bg}


\author[M.~Manev]{Mancho Manev}
\address{%
  Faculty of Mathematics and Informatics \\
  University of Plovdiv  \\
  236, Bulgaria Blvd. \\
  Plovdiv 4003, Bulgaria
}
\email{mmanev@pu.acad.bg, mmanev@yahoo.com}

\author[K.~Gribachev]{Kostadin Gribachev}
\address{%
  Faculty of Mathematics and Informatics \\
  University of Plovdiv  \\
  236, Bulgaria Blvd. \\
  Plovdiv 4003, Bulgaria
}
\email{costas@pu.acad.bg}




\begin{abstract}
The basic class of the non-integrable almost complex manifolds
with a pair of Norden metrics are considered. The interconnections
between corresponding quantities at the transformation between the
two Levi-Civita connections are given. A 4-parametric family of
4-dimensional quasi-K\"ahler manifolds with Norden metric is
characterized with respect to the associated Levi-Civita
connection.
\end{abstract}

\maketitle

\section*{Introduction}

It is a fundamental fact that on an almost complex manifold with
Hermitian metric (almost Hermitian manifold), the action of the
almost complex structure on the tangent space at each point of the
manifold is an isometry. There is another kind of metric, called a
Norden metric or a $B$-metric on an almost complex manifold, such
that action of the almost complex structure is an anti-isometry
with respect to the metric. Such a manifold is called an almost
complex manifold with Norden metric \cite{GaBo} or with $B$-metric
\cite{GaGrMi}. See also \cite{GrMeDj} for generalized
$B$-manifolds. It is known \cite{GaBo} that these manifolds are
classified into eight classes.

The basic class of the non-integrable almost complex manifolds
(the quasi-K\"ahlerian ones) with Norden metric is considered in
\cite{MekMan}. Its curvature properties are studied and the
isotropic K\"ahler type (a notion from \cite{GRMa}) of the
investigated manifolds is introduced and characterized
geometrically \cite{MekMan}. A 4-parametric family of
4-dimensional quasi-K\"ahler manifolds with Norden metric is
constructed on a Lie group in \cite{GrMaMe1}. This family is
characterized geometrically and there is given the condition for
such a 4-manifold to be isotropic K\"ahlerian.

In the present paper we continue studying the quasi-K\"ahler
manifolds with Norden metric. The purpose is to get of the
interconnections between corresponding quantities at the
transformation between the two Levi-Civita connections. Moreover,
characterizing of the known example from \cite{GrMaMe1} with
respect to the associated Levi-Civita connection.

\vskip 2pc


\section{Almost complex manifolds with Norden metric}\label{sec_1}

\subsection{Preliminaries}\label{sec-prelim}

Let $(M,J,g)$ be a $2n$-dimensional almost complex manifold with
Norden metric, i.e. $J$ is an almost complex structure and $g$ is
a metric on $M$ such that $J^2X=-X$, $g(JX,JY)=-g(X,Y)$ for all
differentiable vector fields $X$, $Y$ on $M$, i.e. $X, Y \in
\X(M)$.

The associated metric $\tilde{g}$ of $g$ on $M$ given by
$\tilde{g}(X,Y)=g(X,JY)$ for all $X, Y \in \X(M)$ is a Norden
metric, too. Both metrics are necessarily of signature $(n,n)$.
The manifold $(M,J,\tilde{g})$ is an almost complex manifold with
Norden metric, too.

Further, $X$, $Y$, $Z$, $U$ ($x$, $y$, $z$, $u$, respectively)
will stand for arbitrary differentiable vector fields on $M$
(vectors in $T_pM$, $p\in M$, respectively).

The Levi-Civita connection of $g$ is denoted by $\nabla$. The
tensor filed $F$ of type $(0,3)$ on $M$ is defined by
\begin{equation}\label{F}
F(X,Y,Z)=g\bigl( \left( \nabla_X J \right)Y,Z\bigr).
\end{equation}
It has the following symmetries
\begin{equation}\label{F-prop}
F(X,Y,Z)=F(X,Z,Y)=F(X,JY,JZ).
\end{equation}

Further, let $\{e_i\}$ ($i=1,2,\dots,2n$) be an arbitrary basis of
$T_pM$ at a point $p$ of $M$. The components of the inverse matrix
of $g$ are denoted by $g^{ij}$ with respect to the basis
$\{e_i\}$.

The Lie form $\theta$ associated with $F$ is defined by $
\theta(z)=g^{ij}F(e_i,e_j,z)$.

A classification of the considered manifolds with respect to $F$
is given in \cite{GaBo}. All eight classes of almost complex
manifolds with Norden metric are characterized there according to
the properties of $F$. The three basic classes are given as
follows
\begin{equation}\label{class}
\begin{array}{l}
\W_1: F(X,Y,Z)=\frac{1}{2n} \left\{
g(X,Y)\theta(Z)+g(X,Z)\theta(Y)\right. \\[4pt]
\phantom{\mathcal{W}_1: F(X,Y,Z)=\frac{1}{4n} }\left.
    +g(X,J Y)\theta(J Z)
    +g(X,J Z)\theta(J Y)\right\};\\[4pt]
\W_2: \mathop{\s} \limits_{X,Y,Z}
F(X,Y,J Z)=0,\quad \theta=0;\\[8pt]
\W_3: \mathop{\s} \limits_{X,Y,Z} F(X,Y,Z)=0,
\end{array}
\end{equation}
where $\s $ is the cyclic sum by three arguments.

The special class $\W_0$ of the K\"ahler manifolds with Norden
metric belonging to any other class is determined by the condition
$F=0$.

The only class of the three basic classes, where the almost
complex structure is not integrable, is the class $\W_3$ -- the
class of the \textit{quasi-K\"ahler manifolds with Norden metric}.
Let us remark that the definitional condition from \eqref{class}
implies the vanishing of the Lie form $\theta$ for the class
$\W_3$.                                                                            

\subsection{Curvature properties}\label{sec-curv}

Let $R$ be the curvature tensor field of $\nabla$ defined by
$
    R(X,Y)Z=\nabla_X \nabla_Y Z - \nabla_Y \nabla_X Z -
    \nabla_{[X,Y]}Z$.
The corresponding tensor field of type $(0,4)$ is determined by
$R(X,Y,Z,U)=g(R(X,Y)Z,U)$. The Ricci tensor $\rho$ and the scalar
curvature $\tau$ are defined as usual by
$\rho(y,z)=g^{ij}R(e_i,y,z,e_j)$ and $\tau=g^{ij}\rho(e_i,e_j)$.

It is well-known that the Weyl tensor $W$ on a $2n$-dimensional
pseudo-Rie\-mann\-ian manifold ($n\geq 2$) is given by
\begin{equation}\label{W}
    W=R-\frac{1}{2n-2}\psi_1(\rho)-\frac{\tau}{2n-1}\pi_1,
\end{equation}
where
\begin{equation}\label{psi-pi}
    \begin{array}{l}
      \psi_1(\rho)(x,y,z,u)=g(y,z)\rho(x,u)-g(x,z)\rho(y,u) \\
      \phantom{\psi_1(\rho)(x,y,z,u)}+\rho(y,z)g(x,u)-\rho(x,z)g(y,u); \\
      \pi_1=\frac{1}{2}\psi_1(g)=g(y,z)g(x,u)-g(x,z)g(y,u). \\
    \end{array}
\end{equation}

Moreover, the Weyl tensor $W$ is zero if and only if the manifold
is conformally flat.

\subsection{Isotropic K\"ahler manifolds}\label{sec-iK}
The square norm $\nJ$ of $\nabla J$ is defined in \cite{GRMa} by
\begin{equation}\label{snorm}
    \nJ=g^{ij}g^{kl}
        g\bigl(\left(\nabla_{e_i} J\right)e_k,\left(\nabla_{e_j}
    J\right)e_l\bigr).
\end{equation}

Having in mind the definition \eqref{F} of the tensor $F$ and the
properties \eqref{F-prop}, we obtain the following equation for
the square norm of $\nabla J$
$$
    \nJ=g^{ij}g^{kl}g^{pq}F_{ikp}F_{jlq},
$$
where $F_{ikp}=F(e_i,e_k,e_p)$.

An almost complex manifold with Norden metric $g$ satisfying the
condition $\nJ=0$ is called an \textit{isotropic K\"ahler manifold
with Norden metric} \cite{MekMan}.

\begin{rem}
It is clear, if a manifold belongs to the class $\W_0$, then
it is isotropic K\"ahlerian but the inverse statement is not
always true.
\end{rem}

\vskip 2pc

\section{The transformation $\nabla \rightarrow \tilde{\nabla}$ between the Levi-Civita connections
corresponding to the pair of Norden metrics}

\subsection{The interconnections between corresponding tensors at $\nabla \rightarrow \tilde{\nabla}$}
Let $\nabla$ be the Levi-Civita connection of $g$. Then the
following well-known condition is valid
\begin{equation}\label{LC}
    2g(\nabla_x y,z)=xg(y,z)+yg(x,z)-zg(x,y)
    +g([x,y],z)+g([z,x],y)+g([z,y],x).
\end{equation}
Let $\tilde\nabla$ be the Levi-Civita connection of $\tilde g$.
Then the corresponding condition of \eqref{LC} holds for
$\tilde\nabla$ and $\tilde g$. These two equations imply
immediately
\begin{equation}\label{2nabli}
    g(\nabla_x y-\tilde\nabla_x y,z)=\frac{1}{2}\bigl\{F(Jz,x,y)-F(x,y,Jz)-F(y,x,Jz)\bigr\}.
\end{equation}
Having in mind the properties \eqref{F-prop} of $F$, the condition
\eqref{2nabli} implies the following two equations
\[
    \tilde{F}(x,y,z)=\frac{1}{2}\bigl\{F(Jy,z,x)+F(y,z,Jx)+F(z,Jx,y)+F(Jz,x,y)\bigr\},
\]
\[
    \mathop{\s} \limits_{x,y,z} \tilde{F}(x,y,z)=\mathop{\s} \limits_{x,y,z}
    F(Jx,y,z),
\]
where $\tilde{F}$ is the corresponding structural tensor of $F$
for $(M,J,\tilde{g})$.

The last equation, the definition condition \eqref{class} of the
class $\W_3$ and the properties \eqref{F-prop} of $F$ imply
\begin{thm}\label{W3}
    The manifold $(M,J,\tilde{g})$ is quasi-K\"{a}hlerian with Norden metric if and only if the manifold
    $(M,J,g)$ is also quasi-K\"{a}hlerian with Norden    metric.$\hfill\Box$
\end{thm}

Further, we consider only quasi-K\"{a}hler manifolds with Norden
metric. In this case the condition \eqref{2nabli} and the
definition \eqref{class} of $\W_3$ imply
\begin{equation}\label{nablatildeT}
    \tilde{\nabla}_x y=\nabla_x y+T(x,y),
\end{equation}
where
\begin{equation}\label{T}
    T(x,y)= \left(\nabla_x J\right)Jy+\left(\nabla_y J\right)Jx.
\end{equation}
Hence, by direct computations, we get
\begin{thm}\label{W3-F}
    Let $(M,J,g)$ and $(M,J,\tilde{g})$ be quasi-K\"{a}hler manifolds with Norden metric.
    Then the following equivalent interconnections are valid
    \begin{equation}\label{Ftilde-W3}
    \tilde{F}(x,y,z)=-F(Jx,y,z),
    \end{equation}
    \begin{equation}\label{nablatilde}
    \left(\tilde{\nabla}_x J\right)y=-\left(\nabla_{Jx}
    J\right)Jy.
    \end{equation}
    $\hfill\Box$
\end{thm}

Let $\tilde{R}$ be the curvature tensor of $\tilde{\nabla}$. It is
well-known in the case of transformation of the connection of type
\eqref{nablatildeT} that corresponding curvature tensors have the
following interconnection
\begin{equation}\label{RtildeR_Q}
    \tilde{R}(x,y)z=R(x,y)z+Q(x,y)z,
\end{equation}
where
\begin{equation}\label{Q}
    Q(x,y)z= \left(\nabla_x T\right)(y,z)- \left(\nabla_y T\right)(x,z)
    +T\left(x,T(y,z)\right)-T\left(y,T(x,z)\right).
\end{equation}

By the definition, the corresponding curvature tensor of type
(0,4) is
\begin{equation}\label{Rtilde04}
    \tilde{R}(x,y,z,u)=\tilde{g}\left(\tilde{R}(x,y)z,u\right).
\end{equation}
\begin{thm}\label{W3-Rtilde}
    Let $(M,J,g)$ and $(M,J,\tilde{g})$ be quasi-K\"{a}hler manifolds with Norden metric.
    Then the corresponding curvature tensors at the transformation
    $\nabla \rightarrow \tilde{\nabla}$ satisfy the following equation
    \begin{equation}\label{RtildeR}
    \begin{array}{l}
    \tilde{R}(x,y,z,u)=R(x,y,z,Ju)-\left(\nabla_x F\right)(u,y,z)+
    \left(\nabla_y F\right)(u,x,z) \\[4pt]
    \phantom{\tilde{R}(x,y,z,u)=}
    -g\bigl(\left(\nabla_y J\right)z+\left(\nabla_z J\right)y,
            \left(\nabla_x J\right)Ju+\left(\nabla_u
            J\right)Jx\bigr) \\[4pt]
    \phantom{\tilde{R}(x,y,z,u)=}
    +g\bigl(\left(\nabla_x J\right)z+\left(\nabla_z J\right)x,
            \left(\nabla_y J\right)Ju+\left(\nabla_u
            J\right)Jy\bigr).
    \end{array}
    \end{equation}
\end{thm}
\begin{proof}
Let us denote $T(y,z,u)=g\left(T(y,z),u\right)$. According to
\eqref{T}, the properties of $F$ and the condition $\mathop{\s}
F=0$, we get $T(y,z,u)=F(Ju,y,z)$. Having in mind \eqref{T} and
$\nabla g=0$, we obtain
\[
    g\bigl(\left(\nabla_x T\right)(y,z),u\bigr)=
    \left(\nabla_x T\right)(y,z,u).
\]
Then, by direct computations, we have
\begin{equation}\label{nablaT}
    \begin{array}{l}
    \left(\nabla_x T\right)(y,z,u)-\left(\nabla_y T\right)(x,z,u)=
    \left(\nabla_x F\right)(Ju,y,z)-\left(\nabla_y F\right)(Ju,x,z)\\[4pt]
    \phantom{    \left(\nabla_x T\right)(y,z,u)-\left(\nabla_y T\right)(x,z,u)=}
    +F\bigl(\left(\nabla_x J\right)u,y,z\bigr)
    -F\bigl(\left(\nabla_y J\right)u,x,z\bigr),
    \end{array}
\end{equation}
\begin{equation}\label{TT}
    \begin{array}{l}
    T\left(x,T(y,z)\right)-T\left(y,T(x,z)\right)=
    g\bigl(\left(\nabla_{Ju} J\right)x,
            \left(\nabla_y J\right)Jz+\left(\nabla_z J\right)Jy\bigr) \\[4pt]
    \phantom{T\left(x,T(y,z)\right)-T\left(y,T(x,z)\right)}
    -g\bigl(\left(\nabla_{Ju} J\right)y,
            \left(\nabla_x J\right)Jz+\left(\nabla_z J\right)Jx\bigr).
    \end{array}
\end{equation}
The equations \eqref{RtildeR_Q}, \eqref{Q}, \eqref{Rtilde04},
\eqref{nablaT}, \eqref{TT} imply immediately the interconnection
\eqref{RtildeR} of the statement.
\end{proof}

\subsection{The invariant tensors of the transformation $\nabla \rightarrow \tilde{\nabla}$}
Having in mind the tensor $T$ determined by \eqref{T}, we denote
the corresponding tensor $\tilde{T}$ of $\tilde{\nabla}$. Since we
suppose the conditions \eqref{W3-F} and $\mathop{\s} F=0$, we
obtain
\begin{equation}\label{TtildeT}
    \tilde{T}(x,y)=-T(x,y).
\end{equation}
The last equation and \eqref{nablatildeT} imply
\[
    \tilde{\nabla}_x y+\frac{1}{2}\tilde{T}(x,y)=\nabla_x y+\frac{1}{2}T(x,y).
\]

Let $\tilde{Q}$ be the corresponding tensor for $\tilde{\nabla}$
of the tensor $Q$ determined by \eqref{Q}. Using \eqref{TtildeT},
we obtain the following interconnection between them
\[
    \tilde{Q}(x,y)z=-Q(x,y)z.
\]
Then, from \eqref{RtildeR_Q} we have
\[
    \tilde{R}(x,y)z+\frac{1}{2}\tilde{Q}(x,y)z=R(x,y)z+\frac{1}{2}Q(x,y)z.
\]
Therefore, the following  theorem is valid.
\begin{thm}\label{inv}
    Let $(M,J,g)$ and $(M,J,\tilde{g})$ be quasi-K\"{a}hler manifolds with Norden metric.
    The following two tensors are invariant with respect to the
    transformation $\nabla \rightarrow \tilde{\nabla}$
    \[
    S(x,y)=\nabla_x y+\frac{1}{2}T(x,y),
    \qquad
    P(x,y)z=R(x,y)z+\frac{1}{2}Q(x,y)z.
    \]
 $\hfill\Box$
 \end{thm}

\vskip 2pc
\section{The Lie group as a 4-dimensional $\W_3$-manifold and its characteristics}
\label{sec_3}

In \cite{GrMaMe1} is constructed an example of a 4-dimensional Lie
group equipped with a quasi-K\"{a}hler structure and Norden metric
$g$. There it is characterized with respect to the Levi-Civita
connection $\nabla$. Now, we recall the facts known from
\cite{GrMaMe1} and give their corresponding ones with respect to
the Levi-Civita connection $\tilde{\nabla}$.

\begin{thm}[\cite{GrMaMe1}]\label{G}
Let $(G,J,g)$ be a 4-dimensional almost complex manifold with
Norden metric, where $G$ is a connected Lie group with a
corresponding Lie algebra $\g$ determined by the global basis of
left invariant vector fields $\{X_1,X_2,X_3,X_4\}$; $J$ is an
almost complex structure defined by
\begin{equation}\label{J}
JX_1=X_3,\quad JX_2=X_4,\quad JX_3=-X_1,\quad JX_4=-X_2;
\end{equation}
$g$ is an invariant Norden metric determined by
\begin{equation}\label{g}
\begin{array}{l}
  g(X_1,X_1)=g(X_2,X_2)=-g(X_3,X_3)=-g(X_4,X_4)=1, \\[4pt]
  g(X_i,X_j)=0\quad \text{for}\quad i\neq j \\
\end{array}
\end{equation}
and
\begin{equation}\label{inv-metric}
    g\left([X,Y],Z\right)+g\left([X,Z],Y\right)=0.
\end{equation}
Then $(G,J,g)$ is a quasi-K\"ahler manifold with Norden metric if
and only if $G$ belongs to the 4-parametric family of Lie groups
determined by the conditions
\begin{equation}\label{[]4}
\begin{array}{ll}
[X_1,X_3]=\lambda_2 X_2+\lambda_4 X_4,\quad &
[X_2,X_4]=\lambda_1 X_1+\lambda_3 X_3,\\[4pt]
[X_2,X_3]=-\lambda_2 X_1-\lambda_3 X_4,\quad &
[X_3,X_4]=-\lambda_4 X_1+\lambda_{3} X_2,\\[4pt]
[X_4,X_1]=\lambda_1 X_2+\lambda_4 X_3,\quad &
[X_2,X_1]=-\lambda_2 X_3+\lambda_1 X_4.\\[4pt]
\end{array}
\end{equation}
$\hfill\Box$
\end{thm}

\thmref{W3} and \thmref{G} imply the following
\begin{thm}
The manifold $(G,J,\tilde{g})$ is a quasi-K\"ahler manifold with
Norden metric. $\hfill\Box$
\end{thm}

The components of the Levi-Civita connection $\nabla$ are
determined (\cite{GrMaMe1}) by \eqref{[]4} and
\begin{equation}\label{invLC}
    \nabla_{X_i} X_j=\frac{1}{2}[X_i,X_j]\quad (i,j=1,2,3,4).
\end{equation}
According to \eqref{nablatildeT}, \eqref{T} and \eqref{invLC}, we
receive the following equation for the Levi-Civita connection
$\tilde{\nabla}$
\begin{equation}\label{g-tilde}
    \tilde{\nabla}_{X_i} X_j=\frac{1}{2}\left\{[X_i,X_j]+J[X_i,JX_j]-J[JX_i,X_j]\right\}\quad (i,j=1,2,3,4).
\end{equation}
The components of $\tilde{\nabla}$ are determined by
\eqref{g-tilde} and \eqref{[]4}.

The nonzero components of the tensor $F$ are: \cite{GrMaMe1}

\begin{equation}\label{Fijk}
\begin{split}
-F_{122}&=-F_{144}=2F_{212}=2F_{221}=2F_{234}\\[4pt]
\phantom{-F_{122}}
&=2F_{243}=2F_{414}=-2F_{423}=-2F_{432}=2F_{441}=\lambda_1,\\[4pt]
2F_{112}&=2F_{121}=2F_{134}=2F_{143}=-2F_{211}\\[4pt]
\phantom{2F_{112}}
&=-2F_{233}=-2F_{314}=2F_{323}=2F_{332}=-2F_{341}=\lambda_2,\\[4pt]
2F_{214}&=-2F_{223}=-2F_{232}=2F_{241}=F_{322}\\[4pt]
\phantom{2F_{112}}
&=F_{344}=-2F_{412}=-2F_{421}=-2F_{434}=-2F_{443}=\lambda_3,\\[4pt]
-2F_{114}&=2F_{123}=2F_{132}=-2F_{141}=-2F_{312}\\[4pt]
\phantom{-2F_{114}}
&=-2F_{321}=-2F_{334}=-2F_{343}=F_{411}=F_{433}=\lambda_4,\\[4pt]
\end{split}
\end{equation}
where $F_{ijk}=F(X_i,X_j,X_k)$. We receive the nonzero components
$\tilde{F}(X_i,X_j,X_k)=\tilde{F}_{ijk}$ of the tensor $\tilde{F}$
from \eqref{Ftilde-W3} using \eqref{J} and \eqref{Fijk}. They are
the following
\begin{equation}\label{tildeFijk}
\begin{split}
-2\tilde{F}_{214}&=2\tilde{F}_{223}=2\tilde{F}_{232}=-2\tilde{F}_{241}=-\tilde{F}_{322}\\[4pt]
                 &=-\tilde{F}_{344}=2\tilde{F}_{412}=2\tilde{F}_{421}=2\tilde{F}_{434}=2\tilde{F}_{443}=\lambda_1,\\[4pt]
2\tilde{F}_{114}&=-2\tilde{F}_{123}=-2\tilde{F}_{132}=2\tilde{F}_{141}=2\tilde{F}_{312}\\[4pt]
                &=2\tilde{F}_{321}=2\tilde{F}_{334}=2\tilde{F}_{343}=-2\tilde{F}_{411}=-2\tilde{F}_{433}=\lambda_2,\\[4pt]
-\tilde{F}_{122}&=-\tilde{F}_{144}=2\tilde{F}_{212}=2\tilde{F}_{221}=2\tilde{F}_{234}\\[4pt]
                &=2\tilde{F}_{243}=2\tilde{F}_{414}=-2\tilde{F}_{423}=-2\tilde{F}_{432}=2\tilde{F}_{441}=\lambda_3,\\[4pt]
2\tilde{F}_{112}&=2\tilde{F}_{121}=2\tilde{F}_{134}=2\tilde{F}_{143}=-\tilde{F}_{211}\\[4pt]
                 &=-\tilde{F}_{233}=-2\tilde{F}_{314}=2\tilde{F}_{323}=2\tilde{F}_{332}=-2\tilde{F}_{341}=\lambda_4.\\[4pt]
\end{split}
\end{equation}


The square norm of $\nabla J$ with respect to $g$, defined by
\eqref{snorm}, has the form  \cite{GrMaMe1}
\[
    \nJ=4\left(\lambda_1^2+\lambda_2^2-\lambda_3^2-\lambda_4^2\right).
\]
Then the manifold $(G,J,g)$ is isotropic K\"ahlerian if and only
if
    the condition $\lambda_1^2+\lambda_2^2-\lambda_3^2-\lambda_4^2=0$
    holds. \cite{GrMaMe1}

Analogously, the square norm $\ntJ$ of $\tilde{\nabla} J$ with
respect to $\tilde{g}$ is defined by
\[
    \ntJ=\tilde{g}^{ij}\tilde{g}^{kl}
        \tilde{g}\left(\left(\tilde{\nabla}_{e_i} J\right)e_k,\left(\tilde{\nabla}_{e_j}
    J\right)e_l\right),
\]
where $\tilde{g}^{ij}$ are the components of the inverse matrix of
$\left(\tilde{g}_{ij}\right)$. The last equation has the following
form in terms of $\tilde{\nabla}$
\[
    \ntJ=\tilde{g}^{ij}\tilde{g}^{kl}\tilde{g}^{pq}\tilde{F}_{ikp}\tilde{F}_{jlq}.
\]
Hence by \eqref{tildeFijk} we obtain
\[
    \ntJ=-8\left(\lambda_1\lambda_3+\lambda_2\lambda_4\right).
\]
The last equation implies the following
\begin{prop}\label{prop-iK}
    The manifold $(G,J,\tilde{g})$ is isotropic K\"ahlerian if and only if
    the condition $\lambda_1\lambda_3+\lambda_2\lambda_4=0$
    holds.
\end{prop}
\begin{rem}\label{rem1}
As is known from \cite{GrMaMe1}, $(G,J,g)$ is isotropic
K\"ahlerian if and only if the set of vectors with the coordinates
$(\lambda_1,\lambda_2,\lambda_3,\lambda_4)$ at an arbitrary point
$p\in G$ describes the isotropic cone in $T_pG$ with respect to
$g$. Obviously, $(G,J,\tilde{g})$ is isotropic K\"ahlerian if and
only if the the same set of vectors describes the isotropic cone
in $T_pG$ but with respect to $\tilde{g}$.
\end{rem}
\begin{rem}\label{rem2}
Let us note that two manifolds $(G,J,g)$ and $(G,J,\tilde{g})$ can
be isotropic K\"ahlerian by independent way. For example:
   \begin{enumerate}
    \renewcommand{\labelenumi}{(\roman{enumi})}
    \item
    if $\lambda_3= \lambda_2$ and $\lambda_4=-\lambda_1$,
    then both $(G,J,g)$ and $(G,J,\tilde{g})$ are isotropic K\"ahler
    manifolds;
    \item
    if $\lambda_3= \lambda_1 \neq 0$ and $\lambda_4= \lambda_2$, then
    $(G,J,g)$ is an isotropic K\"ahler manifold but $(G,J,\tilde{g})$ is not isotropic
    K\"ahlerian;
    \item
    if $\lambda_2= \lambda_1 \neq 0$, $\lambda_4=- \lambda_3$ and $|\lambda_1|\neq |\lambda_3|$, then
    $(G,J,g)$ is not an isotropic K\"ahler manifold but $(G,J,\tilde{g})$ is isotropic
    K\"ahlerian.
    \end{enumerate}
\end{rem}

The components $R_{ijks}=R(X_i,X_j,X_k,X_s)$ $(i,j,k,s=1,2,3,4)$
of the curvature tensor $R$ on $(G,J,g)$ are: \cite{GrMaMe1}
\begin{equation}\label{Rijks}
\begin{array}{ll}
    R_{1221}=-\frac{1}{4}\left(\lambda_1^2+\lambda_2^2\right),\quad
    &
    R_{1331}=\frac{1}{4}\left(\lambda_2^2-\lambda_4^2\right),\\[4pt]
    R_{1441}=-\frac{1}{4}\left(\lambda_1^2-\lambda_4^2\right),\quad
    &
    R_{2332}=\frac{1}{4}\left(\lambda_2^2-\lambda_3^2\right),\\[4pt]
    R_{2442}=\frac{1}{4}\left(\lambda_1^2-\lambda_3^2\right),\quad
    &
    R_{3443}=\frac{1}{4}\left(\lambda_3^2+\lambda_4^2\right),\\[4pt]
    R_{1341}=R_{2342}=-\frac{1}{4}\lambda_1\lambda_2,\quad
    &
    R_{2132}=-R_{4134}=\frac{1}{4}\lambda_1\lambda_3,\\[4pt]
    R_{1231}=-R_{4234}=\frac{1}{4}\lambda_1\lambda_4,\quad
    &
    R_{2142}=-R_{3143}=\frac{1}{4}\lambda_2\lambda_3,\\[4pt]
    R_{1241}=-R_{3243}=\frac{1}{4}\lambda_2\lambda_4,\quad
    &
    R_{3123}=R_{4124}=\frac{1}{4}\lambda_3\lambda_4.\\[4pt]
\end{array}
\end{equation}

We get the nonzero components
$\tilde{R}_{ijks}=\tilde{R}(X_i,X_j,X_k,X_s)$, ($i,j,k,s=1,2,3,4$)
of $\tilde{R}$ on $(G,J,\tilde{g})$, having in mind \eqref{Fijk},
\eqref{Rijks}, \eqref{RtildeR}, \eqref{J} as follows:
\begin{equation}\label{tRijks}
\begin{array}{c}
   \tilde{R}_{1221}=-\tilde{R}_{1441}=-\tilde{R}_{2332}=\tilde{R}_{3443}
    =\lambda_1\lambda_3+\lambda_2\lambda_4, \\[4pt]
\begin{array}{c}
    \tilde{R}_{1331}=-\frac{1}{2}\lambda_2\lambda_4,
\quad
    \tilde{R}_{2442}=-\frac{1}{2}\lambda_1\lambda_3,
\\[4pt]
    \tilde{R}_{1234}=\tilde{R}_{1432}
        =\frac{3}{4}\left(\lambda_1\lambda_3+\lambda_2\lambda_4\right), \\[4pt]
\end{array}
\\
\begin{array}{l}
    \tilde{R}_{1241}=\frac{1}{4}\left(4\lambda_1^2+2\lambda_2^2+\lambda_3^2-2\lambda_4^2\right),\\[4pt]
    \tilde{R}_{2132}=\frac{1}{4}\left(2\lambda_1^2+4\lambda_2^2-2\lambda_3^2+\lambda_4^2\right),\\[4pt]
    \tilde{R}_{4134}=\frac{1}{4}\left(-2\lambda_1^2+\lambda_2^2+2\lambda_3^2+4\lambda_4^2\right),\\[4pt]
    \tilde{R}_{3243}=\frac{1}{4}\left(\lambda_1^2-2\lambda_2^2+4\lambda_3^2+2\lambda_4^2\right),\\[4pt]
\end{array}
\\
\begin{array}{l}
    \tilde{R}_{1231}=\tilde{R}_{2142}=-\frac{1}{4}\left(2\lambda_1\lambda_2+3\lambda_3\lambda_4\right),\\[4pt]
    \tilde{R}_{1341}=\tilde{R}_{4124}=-\frac{1}{4}\left(2\lambda_1\lambda_4-3\lambda_2\lambda_3\right),\\[4pt]
    \tilde{R}_{3143}=\tilde{R}_{4234}=-\frac{1}{4}\left(3\lambda_1\lambda_2+2\lambda_3\lambda_4\right),\\[4pt]
    \tilde{R}_{3123}=\tilde{R}_{2342}=\frac{1}{4}\left(3\lambda_1\lambda_4-2\lambda_2\lambda_3\right).\\[4pt]
\end{array}
\end{array}
\end{equation}

The scalar curvature $\tau$ on $(G,J,g)$ is \cite{GrMaMe1}
\[
    \tau=-\frac{3}{2}\left(\lambda_1^2+\lambda_2^2-\lambda_3^2-\lambda_4^2\right),
\]
and for the scalar curvature $\tilde{\tau}$ on $(G,J,\tilde{g})$
we obtain
\begin{equation}\label{taut}
    \tilde{\tau}=5\left(\lambda_1\lambda_3+\lambda_2\lambda_4\right).
\end{equation}
    The manifold $(G,J,g)$ is isotropic K\"ahlerian if and only if it has a zero scalar curvature $\tau$. \cite{GrMaMe1}
\propref{prop-iK} and equation \eqref{taut} imply
\begin{prop}\label{prop-iKt}
    The manifold $(G,J,\tilde{g})$ is isotropic K\"ahlerian if and only if
    it has zero scalar curvature $\tilde{\tau}$.
\end{prop}

\begin{rem}\label{rem3}
Let us note that two manifolds $(G,J,g)$ and $(G,J,\tilde{g})$ can
be scalar flat by independent way.
\end{rem}
Let us recall, that the manifold $(G,J,g)$ has vanishing Weyl
tensor $W$ \cite{GrMaMe1}. Now, let us consider the Weyl tensor
$\tilde{W}$ on $(G,J,\tilde{g})$ determined by analogy to $W$ by
\eqref{W} and \eqref{psi-pi}. Taking into account \eqref{J},
\eqref{g}, \eqref{tRijks} and \eqref{taut}, we receive the
following nonzero components of $\tilde{W}$:
\begin{equation}\label{tWijks}
\begin{array}{l}
    \tilde{W}_{1221}=-\tilde{W}_{1441}=-\tilde{W}_{2332}=\tilde{W}_{3443}=3\tilde{W}_{1234}=3\tilde{W}_{1432}\\[4pt]
    \phantom{\tilde{W}_{1221}}
                    =-\frac{3}{2}\tilde{W}_{1331}=-\frac{3}{2}\tilde{W}_{2442}
    =\lambda_1\lambda_3+\lambda_2\lambda_4.
\end{array}
\end{equation}

According to \propref{prop-iKt} and Equations \eqref{taut},
\eqref{tWijks}, we establish the truthfulness of the following
\begin{thm}
    The following conditions are equivalent for the manifold $(G,J,\tilde{g})$:
    \begin{enumerate}
    \renewcommand{\labelenumi}{(\roman{enumi})}
    \item
    $(G,J,\tilde{g})$ is an isotropic K\"ahler manifold;
    \item
    the Weyl tensor $\tilde{W}$ vanishes;
    \item
    the condition $\lambda_1\lambda_3+\lambda_2\lambda_4=0$
    holds;
    \item
    the scalar curvature $\tilde{\tau}$ vanishes.
    \end{enumerate}
  $\hfill\Box$
\end{thm}



\small{

\end{document}